\documentclass[12pt]{amsart}

\usepackage{amssymb}

\usepackage{enumerate,hyperref}

\usepackage{graphicx}
\usepackage[all]{xy}
\usepackage{xspace}
\makeatletter
\@namedef{subjclassname@2020}{%
  \textup{2020} Mathematics Subject Classification}
\makeatother

\newtheorem{theorem}{Theorem}[section]
\newtheorem{corollary}[theorem]{Corollary}
\newtheorem{lemma}[theorem]{Lemma}
\newtheorem{proposition}[theorem]{Proposition}
\newtheorem{problem}[theorem]{Problem}
\newtheorem{fact}[theorem]{Fact}
\newtheorem{question}[theorem]{Question}

\theoremstyle{definition}

\newtheorem{remark}[theorem]{Remark}

\newcommand{\T}{\mathbb{T}}
\newcommand{\Z}{\mathbb{Z}}

\newcommand{\R}{\mathbb{R}}

\newcommand{\Q}{\mathbb{Q}}

\def\hull#1{\langle#1\rangle}

\DeclareMathOperator{\res}{\hskip-1pt\restriction\hskip-1pt}

\numberwithin{equation}{section}

\begin{document}


\baselineskip=17pt


\title[Precompact Abelian Groups]{Densities and Weights of Quotients of Precompact Abelian Groups}
\author[D. Peng]{Dekui Peng}\thanks{This work is supported by NSFC grant NO. 12271258.}
\email{pengdk10@lzu.edu.cn}
\address{Institute of Mathematics, Nanjing Normal University, Nanjing 210046, China}

\begin{abstract}
The topological group version of the celebrated Banach-Mazur problem asks wether every infinite topological group has a non-trivial separable quotient group.
It is known that compact groups have infinite separable metrizable quotient groups.
However, as dense subgroups of compact groups, precompact groups may admit no non-trivial separable quotient groups, so also no non-trivial metrizable quotient groups.
In this paper, we study the least cardinal $\mathfrak{m}$ (resp. $\mathfrak{n}$) such that every infinite precompact abelian group admits a quotient group with density $\leq \mathfrak{m}$ (resp. with weight $\leq \mathfrak{n}$).
It is shown that if $2^{<\mathfrak{c}}=\mathfrak{c}$, then $\mathfrak{m}=\mathfrak{c}$ and $\mathfrak{n}=2^\mathfrak{c}$.

A more general problem is to describe the set $QW(G)$ of all possible weights of infinite proper quotient groups of a precompact abelian group $G$.
We prove that for every subset $E$ of the interval $[\omega, \mathfrak{c}]$, there exists a precompact abelian group $G$ with $QW(G)=E$.
If $\omega\in E$, then $G$ can be chosen to be pseudocompact.
In an appendix, we give an example to show that a non-totally disconnected locally compact group may admit no separable quotient groups. This answers an open problem posed by Leiderman, Morris and Tkachenko.
\end{abstract}

\maketitle

\section{Introduction}\label{Sec1}
The famous Banach-Mazur Problem asks whether every infinite-dimensional Banach space has a  separable infinite-dimensional  quotient  Banach space.
This problem, of course due to Banach and Mazur, is known to be considered in 1930s.
Although the separable quotient problem for general Banach spaces remains open, positive answers for special classes of Banach
spaces have been obtained, for example, the space $C(X)$ of continuous functions on an infinite compact space $X$ (E. Lacy \cite{La} and H.P. Rosenthal \cite{Ros}) and dual spaces of infinite-dimensional Banach spaces (S.A. Argyros, P. Dodos and V. Kanellopoulos \cite{ADK}).
The separable quotient problem for topological groups asking if every infinite topological group admit a non-trivial (or infinite) separable quotient group has been formulated and studied since 2019, see \cite{GM, LMT, LT1, LT2, Mor}.

It is a well-known result of Gleason, Montgomery-Zippin and Yamabe that compact groups as well as connected locally compact groups (indeed, all connected-by-compact locally compact groups) are projective limits of (separable) Lie groups \cite{MZ}; so they all admit non-trivial quotient Lie groups.
In particular, connected locally compact groups admit metrizable and separable connected non-trivial, hence infinite, quotient groups.
It seems that there is no enough evidence to ensure that profinite groups have infinite metrizable quotient groups because their quotient Lie groups are finite.
However, Leiderman, Morris and Tkachenko proved that profinite groups  indeed have such quotient groups \cite{LMT} by showing that every profinite group admits a continuous homomorphism into a countably infinite product of finite groups with infinite image.
The author would like to note that the result can also be obtained from the following simple observation: let $G$ be an infinite profinite group and $H_1>H_2>...>H_n>....$ a strictly decreasing chain of open normal subgroups of $G$.
The intersection $H=\bigcap_{n=1}^\infty H_n$ is a $G_\delta$ set and a closed normal subgroup in $G$.
Thus $G/H$ is metrizable (and separable).
It is evident that $G/H$ is infinite; otherwise the chain $H_1>H_2>...$ stops at some $n$.

Let us recall that a topological group $G$ is called {\em precompact} if for every identity neighborhood $U$ one can find a finite subset $F$ of $G$ such that $G=FU$.
By a theorem of Weil (see \cite{Weil}), a topological group is precompact if and only if it embeds into a compact group as a dense subgroup.
The family of precompact groups seems to be one of the classes of topological groups closest to that of compact groups.
So we naturally expect that precompact groups also have the property.
Unfortunately, even an abelian precompact group may have no non-trivial separable quotient groups.
Such a group was constructed by Leiderman, Morris and Tkachenko in  \cite{LMT}.
It should be noted that metrizable precompact groups are always separable. So their example also shows that precompact abelian groups may admit no non-trivial metrizable quotient groups.
On the other hand, as a dense subgroup of some compact group,  every infinite precompact (abelian) group $G$ admits a continuous homomorphism $\varphi$ into an infinite metrizable compact group with dense image.
So $G/\ker \varphi$ is isomorphic to a dense subgroup of an infinite metrizable compact group as abstract groups.
Since infinite metrizable compact groups has cardinality $\mathfrak{c}$, one has that $d(G/\ker \varphi)\leq |G/\ker \varphi|\leq \mathfrak{c}$.

\begin{proposition} Every infinite precompact (abelian) group admits a quotient group with density $\leq\mathfrak{c}$.\end{proposition}

Since every topological group $G$ satisfies $w(G)\leq 2^{d(G)}$, we also obtain:

\begin{proposition} Every infinite precompact (abelian) group admits a quotient group with weight $\leq2^\mathfrak{c}$.\end{proposition}

So we arise the following question:

\begin{question} What are the least cardinals $\mathfrak{m}$ and $\mathfrak{n}$ such that every infinite abelian precompact group admits

\begin{itemize}
         \item[(i)] an infinite quotient group with density $\leq \mathfrak{m}$;
         \item[(ii)] an infinite quotient group with weight $\leq \mathfrak{n}$?
\end{itemize}
\end{question}
It follows immediately from the example of Leiderman, Morris and Tkachenko that both $\mathfrak{m}$ and $\mathfrak{n}$ are uncountable.
And by our analysis above, one has $\mathfrak{m}\leq \mathfrak{c}$ and $\mathfrak{n}\leq 2^\mathfrak{c}$.
We especially wonder if the equalities hold.
Note that if we assume {\bf CH}, the continuum hypothesis, then the result of Leiderman-Morris-Tkachenko yields that $\mathfrak{m}=\mathfrak{c}$.
And by carefully analyzing  their construction, one can also obtain that $\mathfrak{n}=2^\mathfrak{c}$ if $\mathfrak{c}^+=2^\mathfrak{c}$.
So, we have:
\begin{proposition}If the generalized continuum hypothesis holds, then $\mathfrak{m}=\mathfrak{c}$ and $\mathfrak{n}=2^\mathfrak{c}$.
\end{proposition}

In the following, we shall show that, in order to obtain the same result, one may only assume that $2^{<\mathfrak{c}}=\mathfrak{c}$,  which is a consequence of Martin's Axiom.

We also discuss the following problem which seems much more complicated:
\begin{problem} Let $G$ be an infinite precompact abelian group. Describe the sets $QD(G)$ and $QW(G)$,
where
\begin{itemize}
\item $QD(G)=\{d(K): K~\mbox{is~an~infinite~proper~quotient~group~of}~G\}$; and
\item $QW(G)=\{w(K): K~\mbox{is~an~infinite~proper~quotient~group~of}~G\}.$
\end{itemize}
\end{problem}

Clearly, for infinite compact abelian groups $G$, one has $QW(G)$ is the interval $[\omega, w(G)]$.
It is an easy excise to check that this also holds for infinite {\em totally minimal abelian groups}, i.e., topological abelian groups $G$ such that every surjective continuous homomorphism with domain $G$ is open.
Note that according to a theorem of Prodanov \cite{Prod}, totally minimal abelian groups are precompact.
As a matter of fact, a stronger version of the result discovered by Prodanov and Stoyanov states that every {\em minimal} abelian group (see Section \ref{Sec4}) is precompact.

We provide several results about the sets $QD(G)$ and $QW(G)$.
In particular, it is shown that each subset of the interval $[\omega, \mathfrak{c}]$ coincides with $QW(G)$ for some precompact abelian group $G$.
Similarly, every subset of $[\omega, \mathfrak{c}]$ containing $\omega$ equals $QW(G)$ for a pseudocompact abelian group $G$.

Our main tool is the duality of precompact abelian groups discovered by Comfort and Ross.
For a precompact abelian group $G$, $G^*$ denotes the group of continuous homomorphisms of $G$ into the 1-dimensional torus, endowed with the pointwise convergence topology.
Although the topological isomorphism $G\cong (G^{*})^{*}$ is known by many authors (see, for example, \cite{HM}), the functor property of $^*$ has been less discussed.
We shall give some basic facts about the functor with proofs in the following.

\subsection{Notation and Terminology}

All topological spaces are assumed to be Hausdorff, unless otherwise stated.
In this paper we only discuss abelian groups, but we are not going to omit this word ``abelian''.
Abelian groups are written additively, including the torus $\T$, so the neutral element of an abelian group is denoted by $0$.
Let $X$ be a subset of a group $G$, the subgroup of $G$ generated by $X$ is denoted by $\hull{X}$.
If $X=\{g\}$, we shall write simply $\hull{g}$ instead of $\hull{\{g\}}$.
A subset $X$ of an abelian group $G$ is called {\em independent} if for every finite subset $\{x_1, x_2, ..., x_k\}$ of $X$ and integers $n_1, n_2,...,n_k$, $n_1x _1+n_2 x_2+...+n_k x_k=0$ implies that $n_1 x_1=n_2 x_2=...=n_k x_k=0$.
One immediately obtain that $\hull{X}$ is a free abelian group if $X$ is an independent set of non-torsion elements.
The {\em rank} of $G$, denoted by $r(G)$, is the cardinality of a maximal independent consisting only elements of infinite and prime power order.
This cardinal does not depend on the choice of the independent subsets; so it is a well-defined cardinal invariant.
If $Y$ is a subset of $X$, the morphism $Y\hookrightarrow X$ always stands for the inclusion.

We denote by $\Z$ the additive group of integers, by $\Q$ the rationals and by $\R$ the reals.
If $G$ is a topological group, then $G_d$ is the group $G$ with discrete topology.
The theory of uniform spaces of A. Weil, developed in \cite{Weil}, made it possible to treat topological groups from these positions, since every topological group has several natural compatible uniform structures.
The completion with respect to the left uniformity (resp. two-sided) of a topological group is called the {\em Weil completion} (resp. {\em Ra\u{\i}kov completion}).
A topological group coincides with its Weil completion (resp.  Ra\u{\i}kov completion) is called {\em Weil complete} (resp.  {\em Ra\u{\i}kov complete}).
The Weil completion of a topological group is not in general a group under the naturally extended operation.
However this does not happen to the Ra\u{\i}kov completion.
If the Weil completion of a topological group is again a topological group, then the two completions automatically coincide.
Precompact groups and topological abelian groups are in this case, since the left uniformity and the two-sided uniformity coincide on a precompact or abelian topological group.
Throughout the paper, we mainly discuss precompact (abelian) groups. So we call the Weil completion (which is also the Ra\u{\i}kov completion) of a precompact (abelian) group only the {\em completion} for convenience. Evidently, the completion of a precompact (abelian) group is a compact (abelian) group.
In the following, we shall denote the completion of $G$ by $\varrho G$.
For $G$ a topological abelian group, a \emph{(continuous) character} means a (continuous) homomorphism of $G$ into the torus $\T$.
For a topological abelian group $G$ which is either compact or discrete, we denote by $\widehat{G}$ the Pontryagin dual group of $G$, i.e., the group $cHom(G, \T)$ of continuous characters of $G$ endowed with the compact-open topology.
The famous and beautiful Pontryagin Duality states that $G$ and $\widehat{\widehat{G}}$ are isomorphic as topological groups.
So, throughout this paper, we shall identify $G$ with $\widehat{\widehat{G}}$; thus every element $g$ in $G$ is a continuous character of $\widehat{G}$ in the sense $g: \chi\to \chi(g)$, for all $\chi\in \widehat{G}$.
Basic properties of Pontryagin Duality can be found in any book on Topological Groups, so we omit them.

Let $G$ be an abelian group which is compact or discrete and $H$ be a subgroup of $\widehat{G}$, we denote by $H^\perp$ the subgroup of $G$ consisting those elements $g$ such that $\chi(g)=0$ for all $\chi\in H$.
One can easily check that $(H^\perp)^\perp=\overline{H}$.
The following result which is easy to check particularly implies that for every continuous character $\chi:G\to \T$, $\hull{\chi}$ is dense in $\widehat{G}$ if and only if $\chi$ is injective.

\begin{lemma}\label{Le:1}Let $G$ be an abelian group which is discrete or compact and $H$ a closed subgroup of $G$. Let also $\chi\in \widehat{G}$ with kernel $N$.
Then
\begin{itemize}
\item [(i)]the sequence
$$\mathbf{0}\to H^\perp\to \widehat{G}\to \widehat{H}\to \mathbf{0}$$
is exact, where $H^\perp\to \widehat{G}$ is the inclusion and $\widehat{G}\to\widehat{H}$ sends $f'$ to the restriction $f'\res_{H}$;
\item[(ii)] the mapping $\widehat{G/H}\to H^\perp: f\to f\circ p$ is a topological isomorphism, where $p:G\to G/H$ is the canonical projection;
\item [(iii)] the closed subgroup $\overline{\hull{\chi}}$ is topological isomorphic to $\widehat{G/N}$ and $\widehat{G}/\overline{\hull{\chi}}\cong \widehat{N}$.
\end{itemize}
\end{lemma}

\begin{proof}The first and second assertions are well-known. To see the third, we let $H=N$ in the sequence in (i).
To see (iii), it suffices to note that $\overline{\hull{\chi}}=(\hull{\chi}^\perp)^\perp=N^\perp$ and apply (i) and (ii).
\end{proof}

For a set $S$, we are going to use $|S|$ to stand for the cardinality of $S$.
We denote by $\omega$ the first infinite ordinal (so it is also a cardinal), by $\mathfrak{c}$ the cardinal of the continuum, i.e., $\mathfrak{c}=2^\omega$.
For $\tau$ an infinite cardinal, $\ln \tau$ is the least cardinal $\mu$ with $2^\mu\geq \tau$.
The {\em successor} of a cardinal $\tau$ is denoted by $\tau^+$; it is the least cardinal greater than $\tau$.
If $\kappa\leq \tau$, then $2^{<\kappa}$ is the supremum of $\{2^\mu:\mu<\kappa\}$.
In particular, $2^{<\mathfrak{c}}=|\{A\subseteq \mathfrak{c}: |A|<\mathfrak{c}\}|$.
Clearly $2^{<\mathfrak{c}}\geq \mathfrak{c}$ and the equality $2^{<\mathfrak{c}}=\mathfrak{c}$ is equivalent to that $2^\kappa=\mathfrak{c}$ for any infinite $\kappa<\mathfrak{c}$.
For a topological space $X$, we denote by $d(X)$ the \emph{density} of $X$, which is the infimum of the cardinals of dense subsets of $X$.
We also denote by $w(X)$ the \emph{weight} of $X$, i.e., the least cardinal $\kappa$ that $X$ has a base of size $\kappa$.
It is clear that $d(X)\leq w(X)$.
For Tychonoff spaces $X$, one has $w(X)\leq 2^{d(X)}$ \cite[Theorem 1.5.7]{Eng}.
Since topological groups are Tychonoff, the above inequality holds for all topological groups.
If $X$ is an infinite precompact group, then it is easy to see that the local weight is equal to $w(X)$.
Thus metrizable precompact groups are exactly precompact groups with countable weight.
 And if $X$ is further assumed to be an infinite compact group, then we have $d(X)=\ln {w(X)}$ and $|X|=2^{w(X)}$ \cite[Corollary5.2.7]{AT}.

\subsection{The Comfort-Ross Duality for Precompact Abelian Groups}

For $G$ an abelian group and $S$ a set of characters of $G$.
The coarsest topology on $G$ making all $\chi\in S$ continuous is called the topology {\em generated by $S$}.
This topology is evidently a group topology, which is Hausdorff if and only if $S$ separates the points of $G$.
Let $\mathcal{T}(G)$ be the set of all  precompact topology on $G$ with the order that $\tau_1\leq \tau_2$ only if $\tau_1$ is coarser than $\tau_2$, and $\mathcal{S}(\widehat{G})$ the poset  of dense subgroups of $\widehat{G}$ ordered by inclusion.
\begin{theorem}\label{Th:CRdual}\cite{CR1}[Comfort-Ross Duality] Let $G$ be an (abstract) abelian group, for every precompact group topology $\tau$, we denote by $(G, \tau)^*$ the group of continuous characters of $(G, \tau)$. Then we have:
\begin{itemize}
 \item[(a)] $(G, \tau)^*$ is a dense subgroup of $\widehat{G_d}$;
 \item[(b)] a subgroup of $\widehat{G_d}$ generates a Hausdorff topology on $G$ if and only if it is dense;
 \item[(c)] the mapping $\tau\to (G, \tau)^*$ is an isomorphism of the posets $\mathcal{T}(G)$ and $\mathcal{S}(G)$, with the inversion maps every dense subgroup $S$ of $\widehat{G_d}$ to the topology on $G$ generated by $S$.
\end{itemize}
\end{theorem}

In the above theorem and the following, $G_d$ denotes the discrete group having the same underlying group with $G$.
Then the Comfort-Ross duality yields that $G^*$ is a dense subgroup of $\widehat{G_d}$.
The topology of $G^*$ inherited from $\widehat{G_d}$ is therefore the topology of pointwise convergence.
With this topology, $G^*$ is again a precompact group, and is called the {\em Comfort-Ross dual group} (or briefly {\em C-R dual group}) of $G$.

One of the most important consequences of Comfort-Ross duality is the following:

\begin{theorem}Let $G$ be a precompact abelian group. Then the {\em evaluation} $\Psi: G\to (G^*)^*, g\to \Psi_g$ is a topological isomorphism, where $\Psi_g$ is the morphism $G^*\to \T$ sending $\chi$ to $\chi(g)$ for all $\chi\in G^*$.\end{theorem}

The above theorem allows us to identify $G$ with the bi-dual group $(G^*)^*$.

\begin{proposition}\label{Prop:wei}For every precompact abelian group $G$ we have $w(G)=|G^*|$.\end{proposition}
\begin{proof}Let $H=\varrho G$, then we obtain $w(G)=w(H)=|\widehat{H}|=|G^*|$.\end{proof}

\begin{proposition}\label{Prop:May}Let $G$ be a precompact abelian group. The following assertions hold:
\begin{itemize}
  \item[(i)] a subgroup $H$ of $G$ is dense if and only if $H$ separates the points of $G^*$;
  \item[(ii)] a subgroup $H$ of $G$ is dense if and only if the topology on $G^*$ generated by $H$ is Hausdorff;
  \item[(iii)]  $G$ has density less than or equal to $\kappa$ if and only if $G^*$ accepts a coarser Hausdorff group topology of weight less than or equal to $\kappa$, where $\kappa$ is an infinite cardinal.
  \end{itemize}
\end{proposition}
\begin{proof}
Let us denote the underlying group of $G^*$ by $K$.
Then $G=(G^*)^*$ is a dense subgroup of $\widehat{K}$.
So $H$ is dense in $G$ if and only if $H$ is dense in $\widehat{K}$; this is equivalent to that $H$ generates a Hausdorff group topology on $K$, by Theorem \ref{Th:CRdual} (b).
Thus we proves (ii) and also (i).

To see (iii), let first $H$ be a dense subgroup of $G$ with cardinality $d(G)$.
The topology on $K$ generated by $H$ is coarser than the original topology of $G^*$, and is of weight $d(G)$ by Proposition \ref{Prop:wei}.
So if $d(G)\leq \kappa$ then $G^*$ admits a coarser Hausdorff group topology of weight $\leq \kappa$.
Conversely, if $G^*$ accepts such a topology $\sigma$, then the C-R dual group $(K, \sigma)^*$ is a dense subgroup of $\widehat{K}$ contained in $G$.
So it is also  a dense subgroup of $G$.
Applying Proposition \ref{Prop:wei} again, this subgroup has cardinality $\leq \kappa$.
\end{proof}

Fix a continuous homomorphism $f: A\to B$ between precompact abelian groups.
Let us first consider $f:A_d\to B_d$ as the homomorphism between discrete groups.
Then the Pontryagin dual mapping $\widehat{f}: \widehat{B_d}\to \widehat{A_d}, \varphi\to\varphi\circ f$ is a continuous homomorphism between compact groups.
Recall that $B^*$ (resp. $A^*$) is a dense subgroup of $\widehat{B_d}$ (resp. $\widehat{A}_d$).
It follows from the continuity of $f:A\to B$ that the restriction $f^*:=\widehat{f}\res_{B^*}$ sends $B^*$ into $A^*$.
So we get a continuous homomorphism $f^*$ of $B^*$ to $A^*$.
And it is very easy to see that $^*$ is a contravariant functor of the category $\mathbf{PCA}$ of precompact abelian groups (and continuous homomorphisms) to itself.
We omit the proof here.

Moreover, let us see that if we identify $G$ with $(G^*)^*$ by the evaluation for all precompact abelian groups $G$, then $(f^*)^*=f$.
Take $\Psi_a\in (A^*)^*$, then $(f^*)^*(\Psi_a)=\Psi_a\circ f^*$.
So for any $\chi\in B^*$, one has
$$(f^*)^*(\Psi_a)(\chi)=\psi_a(f^*(\chi))=\psi_a(\chi\circ f)=\chi(f(a))=\Psi_{f(a)}(\chi).$$
Thus we have that $(f^*)^*: (A^*)^*\to (B^*)^*$ sends $\Psi_a$ to $\Psi_{f(a)}$ for any $a\in A$.
In other words, $(f^*)^*=f$.
So we have:

\begin{theorem}The  contravariant functor $^*$ is a dual isomorphism of the category $\mathbf{PCA}$  to itself.\end{theorem}

It is well-known and easy to see that in the category $\mathbf{PCA}$, monomorphisms are exactly injective continuous homomorphisms and epimorphisms are exactly continuous homomorphisms with dense image.
So we have the following immediately.
\begin{corollary}\label{Coro:epmo}A continuous homomorphism $f$ between precompact abelian groups is injective if and only if $f^*$ has dense image.\end{corollary}

Recall that a {\em topological group embedding} $\varphi: G\to H$ is a continuous homomorphism such that the restriction to its image $G\to \varphi(G)$ is a topological isomorphism.
In other words, a topological group embedding is a composition of topological isomorphism with an inclusion.
Continuous onto homomorphisms in $\mathbf{PCA}$ are epic, so the above corollary shows that their dual mapping are injective.
We shall see that the dual mappings are indeed topological group embeddings.
\begin{proposition}\label{Prop:1.9}A morphism $G\to H$ in $\mathbf{PCA}$ is a continuous isomorphism if and only if its dual $H^*\to G^*$ is a topological group embedding with dense image.\end{proposition}
\begin{proof}
If the morphism $\varphi: G\to H$ is a continuous homomorphism, then the topology $\tau$ with open sets of the form $\varphi^{-1}(U)$, where $U$ runs over all open sets in $H$, is coarser than the original topology on $G$.
And $\varphi$, considered as the homomorphism of $(G, \tau)\to H$, is a topological isomorphism.
So $\varphi^*$ is exactly a topological isomorphism of $H^*$ to its image $(G, \tau)^*$, which is a dense subgroup of $G^*$, by Proposition \ref{Prop:May} (ii) (or by Corollary \ref{Coro:epmo}).

If $\varphi^*$ is a topological group embedding of $H^*$ into $G^*$ with dense image, then one may identity $H^*$ as a dense subgroup of $G^*$ with $\varphi^*: H^*\to G^*$ the inclusion.
So $\varphi=(\varphi^*)^*$ is a continuous homomorphism between precompact abelian groups with the same underlying group.
In other words, $\varphi$ is a continuous isomorphism.
\end{proof}

\begin{proposition}\label{Prop:exact}Let $f:G\to H$ be a continuous homomorphism of precompact abelian groups.
Then $f$ is surjective if and only if $f^*$ is a topological group embedding.
Moreover, under this case, $f$ is open if and only if $f^*(H^*)$ is closed in $G^*$.
\end{proposition}
\begin{proof}
First we assume that $f$ is surjective. By Corollary \ref{Coro:epmo}, $f^*$ is injective.
Therefore $f^*$ is the composition $j\circ f_0$ of a continuous isomorphism $f_0: H^*\to f^*(H^*)$ and the inclusion $j: f^*(H^*)\to G^*$.
$$\xymatrix{
 &G^* \\
 &H^* \ar[u]^{f^*} \ar[r]_{f_0}  &f^*(H^*)\ar[ul]_{j}
 }$$

Then $f=(f^*)^*=f_0^*\circ j^*$ and, by Proposition \ref{Prop:1.9}, $f_0^*$ is a topological group embedding with dense image.
See the following commutative diagram.
$$\xymatrix{
 &G\ar[d]_f \ar[dr]^{j^*}\\
 &H &(f^*(H^*))^*\ar[l]^{f_0^*}
 }$$

On the other hand, since $f$ is surjective, so is $f_0^*$.
Then $f_0^*$ is a topological isomorphism, and therefore, so is $f_0$.
So $f^*$ is a topological group embedding.

To see the other direction,  we recall the widely known fact that every continuous character of a subgroup of a precompact abelian group extends continuously over the whole group.
Hence, if $f^*:H^*\to G^*$ is a topological group embedding, then for each $h\in (H^*)^*=H$, there exists $g\in (G^*)^*=G$ such that $g\circ f^*=h$, i.e., $f(g)=h$.
So $f$ is surjective.

Now consider the second assertion and assume that $f: G\to H$ is surjective and open.
By the above argument, $f^*:H^*\to G^*$ is a topological group embedding.
Let $K$ be the closure of $f^*(H^*)$ in $G^*$, then $f^*= j_K\circ f_K$, where $f_K$ is exactly $f^*$, but its codomain is considered to be $K$, and $j_K:K\to G^*$ the inclusion.

$$\xymatrix{
 &G^* \\
 &H^* \ar[u]^{f^*} \ar[r]_{f_K}  &K\ar[ul]_{j_K}
 }$$

Dualize the above commutative diagram we have the following one.

$$\xymatrix{
 &G\ar[d]_f \ar[dr]^{j_K^*}\\
 &H &K^*\ar[l]^{f_K^*}
 }$$

Since $f_K$ is a topological group embedding with dense image, $f_K^*$ is a continuous isomorphism by Proposition \ref{Prop:1.9}.
It is evident that $j_K^*$ is surjective.
Then for every open subset $U$ in $K^*$, $f_K^*(U)=f((j_K^*)^{-1}(U))$ is open in $H$.
In other words, $f_K^*$ is a topological isomorphism.
Then, also is $f_K$, i.e, $f^*(H^*)=K$ is closed.

Conversely, suppose that $f$ is surjective and $f^*$ has closed image.
Let $N$ be the kernel of $f$.
Then $f$ induces a continuous isomorphism $\widetilde{f}:G/N\to H$ such that $f=\widetilde{f}\circ \pi$, where $\pi:G\to G/N$ is the canonical projection.
$$\xymatrix{
 &G\ar[d]_\pi\ar[dr]^{f}\\
 &G/N\ar[r]_{\widetilde{f}} &H
 }$$

By our result of the first part, all arrows in the dual diagram are topological group embeddings and the image of $\widetilde{f}^*$ is dense in $(G/N)^*$.
$$\xymatrix{
 &G^* \\
 &(G/N)^* \ar[u]^{\pi^*} & H^*\ar[l]^{\widetilde{f}^*}\ar[ul]_{f^*}
 }$$
Then $\widetilde{f}^*$ must be surjective since $f^*(H^*)$ is closed in $G^*$.
Therefore, $\widetilde{f}^*$ is a topological isomorphism and so is $\widetilde{f}$.
Then $f$ is open.
\end{proof}

The above proposition implies the following result, which will be used frequently in this paper, immediately.
\begin{corollary} If $H$ is a closed subgroup (resp. quotient group) of a precompact abelian group $G$, then $H^*$ is a quotient group (resp. closed subgroup) of $G^*$.\end{corollary}

Recall that the Bohr topology for an abstract group is the finest precompact topology.
An abelian group $G$ endowed with the Bohr topology  has the following properties \cite[Chapter 9]{AT}: (i) every subgroup of $G$ is closed and has its Bohr topology; and (ii) every quotient group of $G$ has the Bohr topology.

\begin{proposition}A precompact abelian group $G$ is compact if and only if $G^*$ carries the Bohr topology.\end{proposition}
\begin{proof} By Theorem \ref{Th:CRdual}, $G^*$ has its Bohr topology  if and only if $G^{**}\cong G$ is the largest dense subgroup of $\widehat{G^*_d}$, where $G^*_d$ is the discrete group $G^*$.
The largest dense subgroup is of course $\widehat{G^*_d}$ itself, which is compact.\end{proof}

For $\kappa$ an infinite cardinal, a non-empty subset $U$ of a topological space $X$ is called a $G_\kappa$ set if $U$ is the intersection of (at most) $\kappa$ many open subsets of $X$.
When $\kappa=\omega$, $G_\kappa$ sets are known as $G_\delta$ sets.
A (normal) subgroup of a topological group of form $G_\kappa$ is called a \emph{$G_\kappa$ (normal) subgroup}.
For a topological group $H$ with a dense subgroup $G$ and an infinite cardinal number $\kappa$, we call that $G$ is \emph{$G_\kappa$-essential} in $H$ if $G$ non-trivially intersects non-trivial closed $G_\kappa$ normal subgroup of $H$.

\begin{corollary}\label{Coro:den}For a precompact abelian group $G$ and an infinite cardinal $\kappa$ such that $|G|\geq\kappa$, the following conditions are equivalent:

\begin{itemize}
  \item[(a)] $d(G)\geq\kappa$;
  \item[(b)] for every subset $X$ of $G$ with $|X|< \kappa$, there exists a non-trivial continuous character $\chi$ of $G$ such that $X\subseteq \ker \chi$;
  \item[(c)] $G^*$ is a $G_\mu$-essential subgroup of $\widehat{G_d}$ provided $\mu<\kappa$.
  \end{itemize}
  \end{corollary}

\begin{proof}We first note that by the theorem of Comfort and Ross, $G^*$ is a dense subgroup of $\widehat{G_d}$; and by Proposition \ref{Prop:wei}, $G^*$ has weight $\geq \kappa$.

(a) $\Rightarrow$ (b). Since $d(G)\geq\kappa$ and $|X|<\kappa$, the subgroup $\hull{X}$ of $G$ is not dense.
Let $H$ be the closure of $\hull{X}$ and take a continuous character $f$ of $G/H$.
Then $\chi:=f\circ \pi$ is the desired character of $G$, where $\pi: G\to G/H$ is the canonical projection.

(b) $\Rightarrow$ (c). Let $N$ be a closed $G_\mu$ subgroup of $\widehat{G_d}$.
Then $\widehat{G_d}/N$ has weight $\leq \mu$.
So the subgroup $N^\perp$ of $G_d$ has cardinality $\leq \mu$, since $\widehat{N^\perp}\cong \widehat{G_d}/N$ by Lemma \ref{Le:1}.
By (b), there exists a non-trivial continuous character $\chi$ of $G$ such that $N^\perp$ is contained in $\ker \chi$.
Then $\chi\in N$.
As a continuous character of $G$, $\chi$ in also in $G^*$.
So we have $G^*\cap N\neq \emptyset$.

(c) $\Rightarrow$ (a).
It suffices to show that $G$ does not have a dense subgroup of cardinality $<\kappa$.
From now on, we only discuss Hausdorff topological groups.
By Proposition \ref{Prop:May} (iii), this is equivalent to verify that $G^*$ does not admit a continuous isomorphism onto a topological group with weight $<\kappa$.
Suppose for the contrary, let $f: G^*\to H$ be  a continuous isomorphism with $w(H)=\mu<\kappa$.
Then $f$ admits a continuous extension $f': \widehat{G_d}\to \varrho H$  over their completions, which is a group homomorphism.
Note that $\ker f'$ is a non-trivial closed $G_\mu$ subgroup of $\widehat{G_d}$ since $w(\varrho H)=w(H)=\mu <\kappa\leq w(\widehat{G_d})$.
So our assumption implies that $G^*\cap \ker f'\neq\mathbf{0}.$
On the other hand, $f'\res_{G^*}=f$ is injective, so has trivial kernel.
In other words, $G^*\cap \ker f'=\mathbf{0}.$
This produces a contradiction.\end{proof}

Note that if $|G|< \kappa$, then $(a)$ and $(b)$ can never hold.
However, in the case $G^*=\widehat{G_d}$, i.e., when $G$ carries the Bohr topology, (c) is true.
That is the reason why we assume $|G|\geq \kappa$.

\section{The Cardinals $\mathfrak{m}$ and $\mathfrak{n}$}\label{Sec2}

In this section, we mainly discuss $\mathfrak{m}$ and $\mathfrak{n}$.
We shall denote $ln \mathfrak{c}^+$ by $\kappa$ in the following theorem, where $\mathfrak{c}^+$ is the successor of $\mathfrak{c}$; hence $\kappa$ is the least cardinal with the property $2^\kappa>\mathfrak{c}$.
In other words, for every cardinal $\alpha<\kappa$, $2^\alpha\leq \mathfrak{c}$.
So a set of size $\mathfrak{c}$ has exactly $\mathfrak{c}$ many subsets of cardinality strictly less than $\kappa$.

\begin{theorem}\label{Th:1}There exists a free, dense, zero-dimensional subgroup $G$ of $\mathbb{T}^\mathfrak{c}$ of size $\mathfrak{c}$ such that every non-trivial quotient group of $G$ has density $\kappa$.\end{theorem}

\begin{proof}Let $X=\{x_\alpha:\alpha<\mathfrak{c}\}$ be a set of size $\mathfrak{c}$ and $\eta=[X]^{<\kappa}$ the set of all subsets of $X$ of size $<\kappa$.
Then by the above argument on $\kappa$ we have that $|\eta|=\mathfrak{c}$; thus we write $\eta=\{Y_\beta:\beta<\mathfrak{c}\}$.
Let $F(X)$ (resp. $F(Y_\beta)$) be the free abelian group with the free basis $X$ (resp. $Y_\beta$).
Then each $F(Y_\beta)$ is a direct summand of $F(X)$.
Note that $\mathbb{Q}^{(\mathfrak{c})}$ splits in to the form $\bigoplus_{\beta<\mathfrak{c}}Q_\beta$, where each $Q_\beta$ is a copy of $\mathbb{Q}^{(\mathfrak{c})}$.
For each $\beta<\mathfrak{c}$, we fix a homomorphism $f_\beta: F(X)\to Q_\beta\hookrightarrow \mathbb{Q}^{(\mathfrak{c})}$ with kernel $F(Y_\beta)$.
Since $\mathbb{Q}^{(\mathfrak{c})}$ is a subgroup of $\mathbb{T}$, the homomorphisms $f_\beta$ can be considered as continuous characters of the discrete group $F(X)$; so they are elements in $\widehat{F(X)}$.
It is easy to see that $\{f_\alpha:\alpha<\mathfrak{c}\}$ is an independent subset of $\widehat{F(X)}$, so it generates a free subgroup $A$.
Moreover, $A$ separates points of $F(X)$, and therefore it is dense in $\widehat{F(X)}\cong \mathbb{T}^\mathfrak{c}$, according to Proposition \ref{Prop:May} (i).
Let $B$ be any non-trivial closed subgroup of $A$ and $\overline{B}$ the closure of $B$ in $\widehat{F(X)}$.
\vspace{0.2cm}

\noindent{\textbf {Claim 1.}}  $w(\widehat{F(X)}/\overline{B})<\kappa$ and $w(B)=\mathfrak{c}$.
\begin{proof}[Proof of Claim 1.]
Let $g$ be a non-zero element in $B$ and $T$ the closure of $\hull{g}$ in $\widehat{F(X)}$.
Since $A$ is generated by those $f_\beta$, $g=\Sigma_{i=1}^n k_if_{\beta_i}$ for some positive integer $n$; with each $k_i$ is a non-zero integer and $\beta_i<\mathfrak{c}$.
As the image of $f_{\beta_i}$ is in $Q_{\beta_i}$, the sum
$$f_{\beta_1}(F(X))+f_{\beta_2}(F(X))+...+f_{\beta_n}(F(X))$$
is direct.
Then the kernel $N$ of $g$ is exactly $\bigcap_{i=1}^n F(Y_{\beta_i})$, which is obviously of size $<\kappa$.
By lemma \ref{Le:1} , one obtains that $\widehat {F(X)}/T\cong \widehat{N}$.
The well-known fact that $w(\widehat{N})=|N|$ then implies that $w(\widehat{F(X)}/T)<\kappa$.
Since $g\in B$, $T$ is contained in $\overline{B}$.
Thus we have $w(\widehat{F(X)}/\overline{B})<\kappa$.
This implies particularly that $w(A/B)<\kappa\leq \mathfrak{c}$ as $A/B$ is a dense subgroup of $\widehat{F(X)}/\overline{B}$.
Therefore, from the equalities $w(A)=w(B)w(A/B)$ and $w(A)=w(\widehat{F(X)})=\mathfrak{c}$ one obtains that $w(B)=\mathfrak{c}$.
\end{proof}


\noindent{\textbf {Claim 2.}} $B$ is $G_\nu$-essential in $\overline{B}$ for every $\nu<\kappa$.
\begin{proof}[Proof of Claim 2.] By Claim 1, $\widehat{F(X)}/\overline{B}$ has weight less than $\kappa$.
In other words,  $\overline{B}$ is a $G_\lambda$ subgroup of $\widehat{F(X)}$ for some $\lambda<\kappa$.
It is evident that $G_\nu$-essentiality implies $G_\mu$-essentiality provided $\mu\leq \nu$, hence we can only prove the claim for $\lambda\leq \nu<\kappa$.
Under such an assumption on $\nu$, a closed $G_\nu$ subgroup $H$ of $\overline{B}$ is also closed in $\widehat{F(X)}$ of type $G_\nu$.
Since $\widehat{F(X)}\cong \mathbb{T}^{\mathfrak{c}}$ carrying the product topology, $H$ contains all but $\nu$ many factors of $\mathbb{T}^{\mathfrak{c}}$.
In other words, there exists $\beta<\mathfrak{c}$ such that $H$ contains all characters $\chi$ of $F(X)$ with $\ker \chi\supseteq F(Y_\beta)$.
This yields that $$ f_\beta\in A\cap H\subseteq A\cap \overline{B}=B.$$
In other words, $B\cap H$ is non-trivial and hence $B$ is $G_\nu$-essential in $\overline{B}$.
\end{proof}

Now we let $G=A^*$; it is exactly the group $F(X)$ with the topology generated by $A$.
We will see that $G$ is the desired group.
First, if $K$ is a non-trivial quotient group of $G$, then $K^*$ is a proper closed subgroup of $G^*=A$.
According to Claim 1, one has $w(K^*)=\mathfrak{c}\geq \kappa$.
Proposition \ref{Prop:wei} then implies that $|K|=w(K^*)\geq\kappa$.
Moreover, according to Corollary \ref{Coro:den} and Claim 2, the density  of $K$ is at least $\kappa$.
\vspace{0.2cm}

\noindent{\textbf {Claim 3.}} $d(G)\leq\kappa$.

\begin{proof}[Proof of Claim 3.]
If $\kappa=\mathfrak{c}$, then it follows from the equality
$$d(G)\leq |G|=F(X)=\mathfrak{c}.$$
Now assume that $\kappa<\mathfrak{c}$, i.e., $\kappa^+\leq \mathfrak{c}=|G|$.
Then we can apply Corollary \ref{Coro:den}, by replacing $\kappa$ by $\kappa^+$ there.
Let $Y$ be a subset of $X$ of size $\kappa$, of course $<\kappa^+$.
As we have seen in the proof of Claim 1, the kernel of each non-zero $\chi\in A$ has cardinality less than $\kappa$.
So $Y\not\subseteq \ker\chi$.
According to Corollary \ref{Coro:den}, $d(G)<\kappa^+$.
In other words, $d(G)\leq \kappa$.\end{proof}

By Claim 3, $d(K)\leq d(G)\leq \kappa$. In summary one has $d(K)=\kappa$.

Let us check that $G$ is a zero-dimensional dense subgroup of $\mathbb{T}^{\mathfrak{c}}$.
It is clear that $G=A^*$ is a dense subgroup of $\widehat{A_d}$, where $A_d$ is the abstract group $A$ with the discrete topology.
Since $A$ is free and of size $\mathfrak{c}$, $\widehat{A_d}\cong \T^\mathfrak{c}$.
Moreover, every element in $A$ maps $G$ into a dense, torsion-free (hence zero-dimensional) subgroup of the torus $\T$.
So $G$ is a subgroup of a product of zero-dimensional groups, and therefore, is also zero-dimensional.\end{proof}

The above example shows:
\begin{theorem}$\ln \mathfrak{c}^+\leq\mathfrak{m}$.
In particular, if $2^{<\mathfrak{c}}=\mathfrak{c}$, then $\mathfrak{m}=\mathfrak{c}$.\end{theorem}

\begin{proof} The first assertion follows from Theorem \ref{Th:1}.
If $2^{<\mathfrak{c}}=\mathfrak{c}$, then $\ln \mathfrak{c}^+=\mathfrak{c}$.
Hence $\mathfrak{c}\leq \mathfrak{m}\leq \mathfrak{c}$.
\end{proof}

Let $\nu$ be an infinite cardinal. A family $\Omega$ of subsets of $\nu$ is called a {\em almost disjoint} (briefly, {\em AD}) provided: (i) $|S|=\nu$ for all $S\in \Omega$; and (ii) $|S\cap T|<\nu$ for any pair of distinct elements $S, T$ in $\Omega$.
By Zorn's Lemma, every AD family is contained in a maximal one.
It is an old result of Sierpinski that an AD family of $\nu$ with cardinality $\leq \nu$ is not maximal; thus $\nu$ always admits an AD family of cardinality $\nu^+$.



\begin{lemma}\label{Le:ad}\cite[Theorem 1.3]{Kun} If $2^{<\nu}=\nu$, then $\nu$ admits an AD family of size $2^\nu$, the maximum possible number.\end{lemma}

\begin{proposition}\label{Prop:2.3}Under the assumption of $2^{<\mathfrak{c}}=\mathfrak{c}$, the compact group $\mathbb{T}^{\mathfrak{c}}$ has a free dense subgroup $G$ such that $|G|=2^{\mathfrak{c}}$ and every non-trivial closed subgroup of $G$ has the same cardinality.\end{proposition}
\begin{proof}The Pontryagin dual group of $\mathbb{T}^{\mathfrak{c}}$ is the discrete restricted product $\mathbb{Z}^{(\mathfrak{c})}$.
Let us denote the latter group by $K$ with a free basis $X=\{x_i:i\in \mathfrak{c}\}$.
Let also $Y=\{y_j:i\in \mathfrak{c}\}$ be a Hamel basis of the linear space $\mathbb{Q}^{(\mathfrak{c})}$ over $\mathbb{Q}$.
According to Lemma \ref{Le:ad}, there is an AD family $\eta:=\{Y_\alpha:\alpha<2^\mathfrak c\}$ of subsets of $Y$.
Now  for each $\alpha$, one can fix an injective homomorphism $f_\alpha$ of $K$ into $\mathbb{Q}^{(\mathfrak{c})}$ extending a bijection of $X$ onto $Y_\alpha$.
Since $\mathbb{Q}^{(\mathfrak{c})}$ is a subgroup of $\mathbb{T}$, each $f_\alpha$ is a character of $K$.
Now let $G$ be the subgroup of $\widehat{K}$ generated by all the $f_\alpha$.
We shall show that $G$ is free and the closure of every subgroup of $G$ generated by one non-zero element has cardinality $2^\mathfrak c$.
Each $g\in G$ can be represented as $\Sigma_{k=1}^nm_kf_{\alpha_k}$, where $m_k$ are integers.
We assume that $m_k=0$ for no $1\leq k\leq n$.
Let $$C=\bigcup_{1\leq k<k'\leq n}f^{-1}_{\alpha_k}(Y_{\alpha_k}\cap Y_{\alpha_{k'}})\subseteq X.$$
Since $\eta$ is almost disjoint, $|C|<\mathfrak c$.
\vspace{0.2cm}

\noindent{\textbf {Claim.}} $\ker g\subseteq \hull{C}$.

\begin{proof}[Proof of Claim.]
Let $x=y+z\in \ker g$, where $y\in \hull{C}$ and $z\in \hull{X\setminus C}$.
We have to prove that $z=0$.
For $1\leq k\leq n$, let $L_k$ be the linear subspace of $\Q^{(\mathfrak{c})}$ generated by $f_{\alpha_k}(C)$ and $M_k$ be the linear subspace generated by $Y_{\alpha_k}\setminus f_{\alpha_k}(C)$.
Note that for $1\leq k\leq n$, $$Y_{\alpha_k}\cap \bigcup_{1\leq k'\leq n, k'\neq k}Y_{\alpha_{k'}}\subseteq f_{\alpha_k}(C).$$
Hence $$(Y_k\setminus f_{\alpha_k}(C))\cap \bigcup_{1\leq k'\leq n, k'\neq k}Y_{\alpha_{k'}}=\varnothing.$$
In particular, the sum
$$M=M_1+M_2+...+M_n$$
is direct.

Let $$L=L_1+L_2+...+L_n.$$
Then $L$ is generated by $\bigcup_{1\leq k\leq n}f_{\alpha_k}(C)=\bigcup_{1\leq k<k'\leq n}(Y_{\alpha_k}\cap Y_{\alpha_{k'}})$.
So it is easy to check that the sum $L+M$ is also direct.
Note that $g(y)\in L$ and $g(z)\in M$.
So $g(x)=0$ implies that $g(y)=g(z)=0$.
Moreover, it is easy to see that $g(z)=\sum_{k=1}^nm_kf_{\alpha_k}(z)$ and $f_{\alpha_k}(z)\in M_k$.
Since $M$ is the direct sum of those $M_k$, one obtain that
$$m_1f_{\alpha_1}(z)=m_2f_{\alpha_2}(z)=...=m_nf_{\alpha_n}(z)=0.$$
By our assumption, no $m_k$ equals $0$.
Hence each $m_kf_{\alpha_k}$ is injective.
Therefore, $z=0$.
\end{proof}
The Claim particularly implies that $g\neq 0$.
So $G$ is free.
Let  $H$ be the closure of $\hull{g}$ in $G$ and $\overline{H}$  the closure of $H$ in $\widehat{K}$.
Now apply Lemma \ref{Le:1}, we have that $\widehat{K}/\overline{H}\cong \widehat{N}$, where $N=\ker g$.
Note that $w(\widehat{N})=|N|\leq |C|+\omega<\mathfrak{c}$.
By our assumption of $2^{<\mathfrak{c}}=\mathfrak{c}$, one has $|\widehat{N}|\leq 2^{w(\widehat{N})}=\mathfrak{c}$.
Moreover, $G/H$ is a dense subgroup of $\widehat{K}/\overline{H}$. So we have
$$\mathfrak{c}\geq |\widehat{K}/\overline{H}|\geq|G/H|.$$
The above inequality shows that $|H|=2^{\mathfrak{c}}$ since $|G|=2^{\mathfrak{c}}$.
\end{proof}

\begin{theorem} Under the assumption of $2^{<\mathfrak{c}}=\mathfrak{c}$, there exists a dense, free, zero-dimensional subgroup of $\T^{2^{\mathfrak{c}}}$ of size $\mathfrak{c}$ such that $H$ admits no non-trivial quotient group of weight $<2^{\mathfrak{c}}$. \end{theorem}

\begin{proof}Let $G$ be constructed in Proposition \ref{Prop:2.3} and $H=G^*$.
Then $H$ is a dense subgroup of $\T^{2^\mathfrak{c}}$ since $G$ is a free abelian group with cardinality $2^\mathfrak{c}$.
On the other hand, as $G$ is dense in $\T^\mathfrak{c}$, $H$ is free and of size $\mathfrak{c}$ (indeed, $H$ is the group $K$ in Proposition \ref{Prop:2.3} with the topology generated by characters in $G$).
The zero-dimensionality of $H$ follows from the fact that every $\chi\in G$ maps $H$ into a zero-dimensional subgroup of $\T$.
The rest part is immediately from Proposition \ref{Prop:wei}.
\end{proof}
Now we have:
\begin{theorem}If $2^{<\mathfrak{c}}=\mathfrak{c}$, then $\mathfrak{n}=2^\mathfrak{c}.$\end{theorem}

\section{The Sets QW(G) and QD(G)}\label{Sec3}

Now let us consider the problem what can the sets $QW(G)$ and $QD(G)$ be for a precompact abelian group? More precisely, does every set $S$ of infinite cardinals can be $QW(G)$ or $QD(G)$ for some precompact abelian group?

Let us first consider the intervals.
As we noticed, for infinite compact abelian groups $G$, $QW(G)$ always equals the interval $[\omega, w(G)]$.
If we discuss the parallel question for $QD(G)$, we usually do not have the same result.
Since compact groups $K$ satisfy $d(K)=\ln w(K)$, so if there exists an infinite cardinal $\tau$ such that $\tau\neq \ln \kappa$ for any $\kappa$ (for example, $(2^{<\mathfrak{c}}=\mathfrak{c})+\neg \mathbf{CH}$ implies every $\omega<\tau<\mathfrak{c}$ is such a cardinal), then $[\omega, d(G)]\neq QD(G)$ if $d(G)\geq \tau$.
However, when $G$ is not restricted to be compact, $QD(G)$ may be an interval.

\begin{proposition}\label{Prop:3.1}Let $G$ be an uncountable abelian group carrying the Bohr topology. Then $QD(G)=[\omega, |G|]$.\end{proposition}
\begin{proof}First we note that if an abelian group $G$ has the Bohr topology, then all its subgroups are closed.
In particular, the only dense subgroup of $G$ is $G$ itself.
So $d(G)=|G|$.
For any cardinal $\tau$ with $\omega\leq \tau\leq |G|$, there exists a non-trivial subgroup $H$ of $G$ with $|G/H|=\tau$, see \cite{Scott} for a much stronger result.
If this has been verified, then $d(G/H)=|G/H|=\tau$ since quotient group $G/H$ also has its Bohr topology.
\end{proof}
Let us now consider an extremely special case.
Recall that a topological group $G$ is called {\em topologically simple} if $G$ has no closed normal subgroups but $G$ and the trivial group.
Moreover, when $G$ is additionally assumed to be abelian, the notion is equivalent to that every non-trivial subgroup of $G$ is dense in $G$.
So if $G$ is a topologically simple precompact group, then $QD(G)=QW(G)=\varnothing$.

\subsection{Topologically simple precompact abelian groups}

The following lemma is obvious so we omit the proof.
\begin{lemma}\label{Le:simp}For every precompact abelian group $G$, $G$ is topologically simple if and only if $G^*$ is topologically simple.\end{lemma}

\begin{proposition}\label{Prop:simp}An infinite compact abelian group $G$ admits a dense topologically simple subgroup if and only if $w(G)\leq \mathfrak{c}$ and $G$ is connected.
Moreover, for each infinite cardinal $\alpha\leq \mathfrak{c}$, there are exactly $2^{\alpha+w(G)}$ many such subgroups with cardinality $\alpha$.\end{proposition}

\begin{proof} We first prove the necessity. If $H$ is a dense topologically simple subgroup of $G$, then $d(G)=\omega$ because $H$ has a countable dense subgroup.
So $w(G)\leq 2^\omega=\mathfrak{c}$.
According to Lemma \ref{Le:simp}, the infinite group $H^*$ is also topologically simple, so torsion-free.
Note that $H^*$ and $\widehat{G}$ share the same underlying group.
So $\widehat{G}$ is torsion-free. This is equivalent to that $G$ is connected.

The sufficiency follows the second part of our theorem.
Now assume $G$ is connected and $w(G)\leq \mathfrak{c}$, then $\widehat{G}$ is torsion-free and of size $\leq \mathfrak{c}$.
\vspace{0.2cm}

\noindent{\textbf {Claim.}} $\widehat{G}$ admits a monomorphism into the $r(\widehat{G})$-dimensional linear space over $\mathbb{Q}$ such that the image contains a Hamel basis.

\begin{proof}[Proof of Claim]
Let $D$ be the divisible hull of $\widehat{G}$. Since $\widehat{G}$ is torsion-free, $D$ is an $r(\widehat{G})$-dimensional linear space over $\mathbb{Q}$ \cite[Page 107]{Fuc}.
Take a Hamel basis $\{x_\kappa:\kappa<r(\widehat{G})\}$ of $D$.
As $\widehat{G}$ is essential subgroup of $D$, i.e., non-zero subgroups of $D$ meet $\widehat{G}$ non-trivially, the 1-dimensional subspace containing $x_\kappa$ has a non-zero element $y_\kappa$ contained in $\widehat{G}$, for each $\kappa<r(\widehat{G})$.
The subset $\{y_\kappa:\kappa<r(\widehat{G})\}$ is again a Hamel basis.
Hence the inclusion $\widehat{G}\hookrightarrow D$ is a required monomorphism.
\end{proof}

Let $\nu=\alpha+w(G)=\alpha+|\widehat{G}|$.
If $\widehat{G}$ is countable, then $r(\widehat{G})$ is countable and hence $\nu=\alpha=\alpha+r(\widehat{G})$;
if $\widehat{G}$ is uncountable, then $r(\widehat{G})=|\widehat{G}|$ by \cite[Proposition 9.9.20]{AT},
 and therefore $\nu=\alpha+r(\widehat{G})$ again holds.
 In summary, $\nu=\alpha+r(\widehat{G})$.
It is known that the linear space $Q:=\mathbb{Q}^{(\mathfrak{c})}$ has $c^\nu=2^\nu$ many linear subspaces of dimension $\nu$.
Now let $\{L_\kappa:\kappa<2^\nu\}$ be a family of $\nu$-dimensional subspaces of $Q$; and let $L_\kappa=\bigoplus_{\beta<\alpha}L_{\kappa, \beta}$, where each $L_{\kappa, \beta}$ is of dimension $r(\widehat{G})$.

Fix $\kappa$, for every $\beta<\alpha$, we take a monomorphism
$$f_{\kappa, \beta}: \widehat{G}\to L_{\kappa, \beta}(\hookrightarrow L_\kappa\hookrightarrow Q\hookrightarrow \T)$$
 such that the image contains a Hamel basis of $L_{\kappa, \beta}$.
This morphism is then considered as an element in $G$, by Pontryagin duality.
Let $H_\kappa$ be the subgroup of $G$ generated by $\{f_{\kappa, \beta}: \beta<\alpha\}$.
Then every non-zero element $g$ in $H_\kappa$ is a morphism of $\widehat{G}$ to $\T$ with trivial kernel since the image of different $f_{\kappa, \beta}$ is in different $L_{\kappa, \beta}$, and the sum of those $L_{\kappa, \beta}$ is direct.
According to Lemma \ref{Le:1} we obtain that $\hull{g}$ is dense in $G$, so in $H_\kappa$.
Evidently, $|H_\kappa|=\alpha$.
It follows our construction immediately that those $H_\kappa$ are pairwise distinct\footnote{Indeed, it is not hard to check that the sum of all $H_\kappa$ is direct.}.

The last step, we note that the compact group $G$ has cardinality $2^{w(G)}$, so has at most $2^{w(G)+\alpha}$ many subsets (subgroups) of size $\alpha$. Now we are done.
 \end{proof}

We immediately have the dual version of the above result.

\begin{corollary}\label{Coro:simp}An infinite abelian group $G$ admits a topologically simple precompact group topology if and only if it is torsion-free and of size $\leq \mathfrak{c}$.
Moreover, for each infinite cardinal $\alpha\leq \mathfrak{c}$, there are exactly $2^{|G|+\alpha}$ many such topologies making $G$ has weight $\alpha$.\end{corollary}
\begin{proof} Note that $G$ is torsion-free if and only if $\widehat{G}$ is connected, and $|G|=w(\widehat{G})$.
Now the result follows immediately from Theorem \ref{Th:CRdual}, Lemma \ref{Le:simp} and Proposition \ref{Prop:simp}.\end{proof}

Now we have:
\begin{corollary}
For each pair of infinite cardinals $\alpha$ and $\beta$ which are both at most $\mathfrak{c}$, there exists a precompact abelian group $G$ such that $|G|=\alpha$, $w(G)=\beta$, $d(G)=\omega$ and $QW(G)=QD(G)=\varnothing$.
\end{corollary}

\begin{remark}After I finished the first version of this paper, I found that Lemma \ref{Le:simp}, (the former version\footnote{In my earlier version, I didn't count the number of such topologies.} of) Proposition \ref{Prop:simp} and its dual version Corollary \ref{Coro:simp} were proved in the very recent paper \cite{HRT} of Hern\'{a}ndez,  Remus and Trigos-Arrieta.
The proofs are not exactly the same but both based on the theorem of Comfort and Ross.
In their Section 7, they proved the following theorem, which is one of the main result of that paper.\\

\noindent{\bf Theorem.} \emph{An infinite torsion-free abelian group of size $\leq \mathfrak{c}$ admits exactly $2^\mathfrak{c}$ topologies making it to be precompact and topologically simple. Moreover, the topologies can be chosen to let $G$ has weight $\mathfrak{c}$.}\\

Clearly our Corollary \ref{Coro:simp} generalizes this result (by letting $\alpha=\mathfrak{c}$).
It should be noted that our result can be also obtained by modifying their proof.
\end{remark}

\subsection{Every subset of $[\omega, \mathfrak{c}]$ is $QW(G)$ for some precompact $G$}

Let us come back to the general case.
Recall that a topological group is called {\em monothetic} if it admits a dense cyclic subgroup.
It is evident that monothetic groups are abelian.
Topologically simple abelian groups as well as the torus $\T$ are examples of monothetic groups.

\begin{theorem}\label{Th:set1}Let $E$ be a nonempty subset of the interval $[\omega, \mathfrak{c}]$. For each $\kappa\in E$, fix a non-zero cardinal $\mu_\kappa\leq \mathfrak{c}$.
Let us also assume that either $|E|>1$ or $E=\{\kappa\}$ with $\mu_\kappa>1$.
Then there exits a zero-dimensional, monothetic, and free subgroup $G$ of $\mathbb{T}^\mathfrak{c}$ such that $QW(G)=E$ and for each $\kappa\in E$, $G$ has exactly $\mu_\kappa$ many proper quotient groups with weight $\kappa$.\end{theorem}

\begin{proof}For each $\kappa\in E$, let $\{X_{\kappa, \alpha}: \alpha< \mu_\kappa\}$ be a family of sets of cardinality $\kappa$.
The disjoint union of this family is noted by $X_\kappa$; and the disjoint union of all $X_\kappa$ is $X$.
Then $X$ has cardinality $\leq \mathfrak{c}$.
Let $F(X)$ be the free abelian group over $X$.
Similarly, $F(X_{\kappa, \alpha})$ is the subgroup generated by $X_{\kappa, \alpha}$.
The linear space $Q:=\mathbb{Q}^{(\mathfrak{c})}$, which is also a dense zero-dimensional subgroup of $\T$, splits into the form $$Q=\bigoplus_{\kappa\in E}(\bigoplus_{\alpha< \mu_\kappa}Q_{\kappa, \alpha}),$$
where $Q_{\kappa,\alpha}$ is a copy of $Q$.

Now, for $\kappa\in E$, $\alpha< \mu_\kappa$, we fix a homomorphism $f_{\kappa, \alpha}$ of $F(X)$ into $Q_{\kappa,\alpha}$ with kernel $F(X_{\kappa, \alpha})$.
Again, each $f_{\kappa, \alpha}$ is a character of $F(X)$.
The subgroup of $\widehat{F(X)}$ generated by all those $f_{\kappa, \alpha}$ is denoted by $G$.
Let $N$ be a non-trivial closed subgroup of $G$ and $N\neq G$.
\vspace{0.2cm}

\noindent{\textbf {Claim 1.}} $N=\hull{f_{\kappa, \alpha}}$ for some $\kappa\in E$ and $\alpha< \mu_\kappa$.

\begin{proof}[Proof of Claim 1.] Suppose that $g=\sum_{i=1}^n m_if_{\kappa_i, \alpha_i}\in N$, where $n\geq 1$, $m_i\neq 0$, and $(\kappa_i, \alpha_i)\neq (\kappa_j, \alpha_j)$ if $1\leq i<j\leq n$.
We first show that $n=1$.
Assume, for the contrary, $n\geq 2$.
Since the sum of the images of all those $f_{\kappa_i, \alpha_i}$ are direct, we have
$$\ker g=\bigcap_{i=1}^n \ker (m_if_{\kappa_i, \alpha_i})=\bigcap_{i=1}^n F(X_{\kappa_i, \alpha_i})=\mathbf{0}.$$
Thus, $g$ is injective.
According to Lemma \ref{Le:1}, we obtain that  $\hull{g}$ is dense in $\widehat{F(X)}$, so also in $G$.
This contradicts to that $N\neq G$.

The above analysis implies that $N=\hull{m f_{\kappa, \alpha}}$ for some non-zero integer $m$ and $\kappa\in E; \alpha\in \mu_{\kappa}$.
Notice that $N^\perp=\ker (m f_{\kappa, \alpha})=\ker f_{\kappa, \alpha}=M^\perp$, where $M$ is the closure of $\hull{f_{\kappa, \alpha}}$ in $\widehat{F(X)}$.
So we have $$N=(N^\perp)^\perp=(M^\perp)^\perp=M.$$
Therefore, $f_{\kappa, \alpha}\in N$.
In other words, $m=1$.
\end{proof}

\noindent{\textbf {Claim 2.}} For each $\kappa\in E$ and $\alpha<\mu_\kappa$, $\hull{f_{\kappa, \alpha}}$ is closed in $G$.

\begin{proof}[Proof of Claim 2.] Let $M$ be the closure of $\hull{f_{\kappa, \alpha}}$.
By Claim 1, $M$ equals to either $\hull{f_{\kappa, \alpha}}$ or $G$.
If $M=G$, i.e., $\hull{f_{\kappa, \alpha}}$ is dense in $G$, therefore, also dense in $\widehat{F(X)}$, then $f_{\kappa, \alpha}: F(X)\to \T$ would be injective.
This induces a contradiction,\end{proof}

Since there is a natural one-to-one corresponding between non-trivial proper closed subgroups of $G$ and non-trivial proper quotient groups of $G$, we have to check that $G/\hull{f_{\kappa, \alpha}}$ has weight $\kappa$, where $\kappa\in E$ and $\alpha<\mu_\kappa$.
It follows from Lemma \ref{Le:1} that $\widehat{F(X)}/ \overline{\hull{f_{\kappa, \alpha}}}$ is topologically isomorphic to $\widehat{F(X_{\kappa, \alpha})}$.
Since the latter group has weight $\kappa$ and $G/\hull{f_{\kappa, \alpha}}$ is dense in  $\widehat{F(X)}/\overline{\hull{f_{\kappa, \alpha}}}$, the weight of $G/\hull{f_{\kappa, \alpha}}$ is $\kappa$ as well.

The last step, let us show that $G$ is a zero-dimensional, monothetic, free subgroup of $\T^\mathfrak{c}$.
It should be noted that in Claim 1, we indeed proved that every element $g=\sum_{i=1}^n m_if_{\kappa_i, \alpha_i}$ in $G$ with $n\geq 2$ and $m_i\neq 0$ generates a dense subgroup of $G$. In particular, $g\neq 0$. So $G$ is free and is monothetic.
It remains to check that $\dim(G)=0$.
Note that $G$ separates points of $F(X)$, so $G$ is a dense subgroup of $\widehat{F(X)}$ and the C-R dual group $G^*$ of $G$ is algebraically isomorphic to $F(X)$.
Since $\widehat{F(X)}$ is a torus of dimension $\leq \mathfrak{c}$, $G$ is a subgroup of $\T^\mathfrak{c}$.
Take $g\in G$ and $x\in G^*(=F(X))$, by the construction of $G$, we know that $x(g)=g(x)\in Q$.
So $G$ is a subgroup of $Q^\mathfrak{c}$ and therefore, zero-dimensional.
\end{proof}

\begin{remark}We must ensure that either $|E|>1$ or $E$ is the singleton $\{\kappa\}$ with $\mu_\kappa>1$ in Theorem \ref{Th:set1}.
If there were no such a restriction, we would get a trivial group by repeating the operations in the above proof.
So it is natural to ask can we find a precompact abelian group $G$ such that $G$ has exactly one non-trivial proper closed subgroup $N$ such that the weight of $G/N$ is the given cardinal $\kappa\in [\omega, \mathfrak{c}]$?
The answer is positive.
Indeed, according to Proposition \ref{Prop:simp}, $\T^\kappa$ has a dense subgroup $K$ which is topologically simple.
If $N$ is a cyclic subgroup of $\T^\kappa$ with order $2$, then $G:=K+N$ is a desired group.
\end{remark}

\subsection{The set $QW(G)$ for $G$ pseudocompact}

Recall that a topological space $X$ is called {\em pseudocompact} if every real-valued function of $X$ has bounded image.
Pseudocompact spaces are deeply studied in the field of General Topology.
In a celebrated paper, Comfort and Ross gave the following nice characterization for a topological group to be pseudocompact.

\begin{theorem}\label{CR2}\cite{CR2}[Comfort and Ross] Pseudocompact groups are precompact and a precompact group $G$ is pseudocompact if and only if $G\cap xN\neq \varnothing$ for any $x\in \varrho G$ and any closed $G_\delta$ subgroup $N$ of $\varrho G$.\end{theorem}
Briefly, we call that $G$ is $G_\delta$-dense in $\varrho G$ if $G$ has the above property.

This characterization of pseudocompact groups gives many useful applications.
Among others, one has:

\begin{fact}\label{Fact} \em{The following facts hold:
\begin{itemize}
  \item [(i)] \cite{LMT} If $G$ is pseudocompact and infinite, then $G$ has infinite compact quotient groups which are also metrizable.
  \item [(ii)] \cite{HM} A precompact abelian group $G$ is pseudocompact if and only if every countable subgroup of its C-R dual group $G^*$ has the Bohr topology.
\end{itemize}}
\end{fact}

By (i), infinite pseudocompact abelian groups $G$ admits infinite metrizable (and therefore separable) quotients, i.e., $\omega\in QW(G)$.
So we want to know whether pseudocompact abelian groups $G$ satisfies the strong property that $QW(G)=[\omega, w(G)]$?
In the following, we give a negative answer to this question.
We shall see the set $QW(G)$ for pseudocompact abelian groups can be quite arbitrary.
The idea of the proof is similar with that of Theorem \ref{Th:set1}.
While, it needs a more complicated construction.

\begin{theorem}\label{Th:set2} Let $E$, $\mu_\kappa$ be defined as in Theorem \ref{Th:set1}. We additionally assume that $\omega\notin E$.
Then there exists a monothetic, pseudocompact, free subgroup $K$ of $\T^{\mathfrak{c}}$ such that $QW(K)=\{\omega\}\cup E$ and for each $\kappa\in E$, $K$ has exactly $\mu_\kappa$ many quotient groups with weight $\kappa$.\end{theorem}

\begin{proof}
We shall produce two dense subgroups $G$ and $H$ of some power ($\leq \mathfrak{c}$) of $\T$ with the following properties:
\begin{itemize}
\item[(i)] $G$ and $H$ are free and $G\cap H=\mathbf{0}$;
\item[(ii)] $H$ is pseudocompact;
\item[(iii)] $G$ is monothetic, $QW(G)=E$ and for each $\kappa\in E$, $G$ has exactly $\mu_\kappa$ many proper quotient groups with weight $\kappa$; and
\item[(iv)] for every closed subgroup $N$ of $K:=G+H$, either $K/N$ is metrizable or $N$ is contained in $G$.
\end{itemize}

If the construction has already been down, we shall see that $K$ is the desired group.
Firstly, (i) means that $K$ is free.
Secondly, since every topological space with a dense pseudocompact subspace is pseudocompact,
$K$, containing $H$ as a dense subgroup, is pseudocompact.
Similarly, the monothetic group $G$ is also a dense subgroup of $K$, so $K$ itself is monothetic as well.
Finally, it remains to check that $QW(K)=E\cup \{\omega\}$ and for $\kappa\in E$, $K$ has exactly $\mu_\kappa$ many quotient groups of weight $\kappa$.
Note that by Fact \ref{Fact} (i), $\omega\in QW(K)$.
So it suffices to prove the two sets, for each $\kappa\in E$,
$$A_{\kappa}:=\{M: M\mbox{~is~a~non-trivial~closed~subgroup~of~}K\mbox{~with~}w(K/M)=\kappa\}$$
and
$$B_{\kappa}:=\{N: N\mbox{~is~a~non-trivial~closed~subgroup~of~}K\mbox{~with~}w(G/N)=\kappa\}$$
coincide.
Since $\omega\notin E$, $\kappa\neq \omega$.
By (iv), if $M\in A_{\kappa}$, then $M\subseteq G$.
As a dense subgroup of $K/M$, $G/M$ has the same weight $\kappa$; thus $M\in B_\kappa$.
Conversely, assume that  $N\in B_{\kappa}$. Then the closure $M$ of $N$ in $K$ satisfies $w(K/M)=w(G/N)=\kappa$.
So $M\in A_{\kappa}$.
Again, $\kappa\neq \omega$, so (iv) implies that $M\subseteq G$, i.e., $M=N$.
Hence, $N\in A_{\kappa}$.
Now we are done.

\noindent{\textbf {The Construction.}}
Let $X_\kappa$, $X$, $F(X)$ and $F(X_\kappa)$ be defined as in the proof of Theorem \ref{Th:set1}.
Then $\widehat{F(X)}\cong \T^{|X|}$ and $|X|\leq \mathfrak{c}$.
Let also $P=Q\oplus R$, where $Q\cong R\cong \Q^{(\mathfrak{c})}$, and
$$Q=\bigoplus_{\kappa\in E}(\bigoplus_{\alpha<\mu_\kappa} Q_{\kappa, \alpha}),$$
where $R_{\kappa, \alpha}\cong \Q^{(\mathfrak{c})}$.

Repeat the construction in Theorem \ref{Th:set1}, for each $\kappa\in E$ and each $\alpha<\mu_\kappa$, we fix a group homomorphism $f_{\kappa, \alpha}: F(X)\to Q_{\kappa, \alpha}$ with kernel $F(X_{\kappa, \alpha})$.
The sequence of inclusions
$$Q_{\kappa, \alpha}\hookrightarrow Q\hookrightarrow P\hookrightarrow\T$$
allows us to identify those $f_{\kappa, \alpha}$ with characters of $F(X)$.
Let also $G$ be the subgroup of $\widehat{F(X)}$ generated by $\bigcup_{\kappa\in E}\{f_{\kappa, \alpha}:\alpha<\mu_\kappa\}$.
As in shown in Theorem \ref{Th:set1}, $G$ is free and satisfies (iii).

Now let us define the other subgroup $H$ of $\widehat{F(X)}$.
We continue splitting $R$ into the form $R=\bigoplus_{\beta<\mathfrak{c}} R_\beta$, where each $R_\beta$ is again a copy of $\Q^{(\mathfrak{c})}$.
Note that $|X|\leq \mathfrak{c}$, so $X$ has exactly $\mathfrak{c}$ many countable subsets.
A character $\chi$ of $F(X)$ is called {\em countably supported} if $\chi(x)=0$ for all but countably many $x\in X$.
It is clear that $F(X)$ admits exactly $\mathfrak{c}$ many countably supported characters.
Let us denote by $\{\chi_\beta: \beta<\mathfrak{c}\}$ the set of all countably supported characters of $F(X)$ and by $Y_\beta$ the set $\{x\in X: \chi_\beta(x)\neq 0\}$ for each $\beta<\mathfrak{c}$.
Obviously, $Y_\beta$ runs over all countable subsets of $X$.
For each $\beta<\mathfrak{c}$, we take a homomorphism $j_\beta: F(X)\to R_\beta$ with kernel $F(Y_\beta)$, the subgroup of $F(X)$ generated by $Y_\beta$.
Each $j_\beta$ is of course a character of $F(X)$ since $R_\beta\hookrightarrow R\hookrightarrow P\hookrightarrow\T$.
Now let $h_\beta=\chi_\beta+j_\beta$ and $H=\hull{h_\beta: \beta<\mathfrak{c}}$.
For each $0\neq x\in F(X)$, there exists a finite subset $Y$ of $X$ such that $x$ in contained in the subgroup $F(Y)$ generated by $Y$.
Take $\beta<\mathfrak{c}$ such that $Y\cap Y_\beta=\emptyset$, then
$$h_\beta(x)=\chi_\beta(x)+j_\beta(x)=j_\beta(x),$$
as $\chi_\beta$ vanishes on $F(X\setminus Y_\beta)\supseteq F(Y)$.
Moreover, since the kernel of $j_\beta$ is $F(Y_\beta)$, which trivially meets $F(Y)$, one has that $j_\beta(x)\neq 0$.
Thus, $H$ separates the points of $F(X)$ and hence $H$ is a dense subgroup of $\widehat{F(X)}$, according to Proposition \ref{Prop:May} (i).

Let $h\in H$ with the form $h=m_1 h_{\beta_1}+m_2 h_{\beta_2}+...+m_n h_{\beta_n}\in H$, where $m_1, m_2, ..., m_n$ are non-zero integers and $\beta_1,\beta_2,...,\beta_n<\mathfrak{c}$.
Denote the subset $X\setminus (\bigcup_{i=1}^nY_{\beta_1})$ of $X$ by $Z$.
The subgroup of $F(X)$ generated by $Z$ is $F(Z)$.

\vspace{0.2cm}
\noindent{\textbf {Claim 1.}} $h(a)\neq 0$ for each $0\neq a\in F(Z)$.
In particular, $H$ is a free abelian group with basis $\{h_\beta:\beta<\mathfrak{c}\}$.

\begin{proof}[Proof of Claim 1.]

Note that $Z\subset X\setminus Y_{\beta_i}$ for $1\leq i\leq n$.
Since $\chi_{\beta_i}(X\setminus Y_{\beta_i})=\{0\}$, one has $\chi_{\beta_i}(F(Z))=\{0\}$.
Hence, for each $0\neq a\in F(Z)$,
$$m_i h_{\beta_i}(a)=m_i\chi_{\beta_i}(a)+m_ij_{\beta_i}(a)=m_ij_{\beta_i}(a)\in R_{\beta_i}.$$
Since the sum $R_{\beta_1}+R_{\beta_2}+...+R_{\beta_n}$ is direct,
$$h(a)=m_1h_{\beta_1}(a)+m_2h_{\beta_2}(a)+...+m_nh_{\beta_n}(a)=0$$
only if
$m_i h_{\beta_i}(a)=m_ij_{\beta_i}(a)=0$ for each $i=1,2,...,n$.
However, none of these elements is equal to $0$.
Indeed, for a given $i$, the kernel of $j_{\beta_i}$ is $F(Y_{\beta_i})$, which trivially meets $F(Z)$.
So $j_{\beta_i}(a)\neq 0$ and hence $m_ih_{\beta_i}(a)\neq 0$.

In summary, we have proved that $h(a)\neq 0$ and particularly $h\neq 0$.
Then the second part of the claim follows.
\end{proof}
Assume that $k\in G\cap H$.
Then by our construction, $$k(F(X))\subseteq Q\cap R=\mathbf{0}.$$
This implies that $k=0$ and hence $G\cap H=\mathbf{0}$.
So (i) is proved.

Now let us prove (ii):

\vspace{0.2cm}
\noindent{\textbf {Claim 2.}} $H$ is pseudocompact.
\begin{proof}[Proof of Claim 2.] According to Fact \ref{Fac} (ii), we have to check that every countable subgroup of $H^*$ carries the Bohr topology.
Note that $H^*$ is exactly the group $F(X)$ with the coarest topology making all $h\in H$ continuous.
Thus, it suffices to show that for every countable subgroup $A$ of $F(X)$ and every homomorphism $\chi: A\to \T$, there exists $h\in H$ extending $\chi$.

Since $A\subseteq F(X)$ is countable and $F(X)$ is the free abelian group over $X$, there exists a countable subset $Y$ such that $A$ is contained in $F(Y)$, the subgroup generated by $Y$.
As $\T$ is divisible, $\chi$ admits an extension $\chi'$ over $F(Y)$.
The mapping of $X$ to $\T$ sending $x$ to $\chi'(x)$ for $x\in Y$ and to $0$ for $x\in X\setminus Y$
induces a countably supported character of $F(X)$.
In other words, there exists $\beta<\mathfrak{c}$ such that the restriction of $\chi_\beta$ to $Y$ is $\chi'\res_Y$ and $Y_\beta\supseteq Y$.
Since $j_\beta$ vanishes on $Y_\beta$, one has $$h_\beta\res_Y=(\chi_\beta+j_\beta)\res_Y=\chi_\beta\res_Y=\chi'\res_Y.$$
This implies that $h_\beta\res_{F(Y)}=\chi'\res_{F(Y)}.$
So $$h_\beta\res_A=\chi'\res_A=\chi.$$
The above equality completes the proof.
\end{proof}

The property (iv) follows from the next claim immediately.

\vspace{0.2cm}
\noindent{\textbf {Claim 3.}} If $N$ is a closed subgroup of $K$ which in not contained in $G$, then $K/N$ is metrizable.
\begin{proof}[Proof of Claim 3.]
Let $N$ be such a subgroup. Then $N$ contains an element $k=g+h$ with $g\in G$ and $0\neq h\in H$.
Then
$$h=m_1h_{\beta_1}+m_2h_{\beta_2}+...+m_bh_{\beta_n}$$
for some non-zero integers $m_1$, $m_2$,...,$m_n$ and cardinals $\beta_1,\beta_2,...,\beta_n$ less than $\mathfrak{c}$.
Retain the symbols in the paragraph before Claim 1, and take a non-zero element $a\in F(Z)$.
Then $h(a)\neq 0$, according to Claim 1.
Note that $h(a)\in R$ and $g(a)\in Q$ with $Q\cap R=\mathbf{0}$, the above inequality implies that $k(a)=h(a)+g(a)\neq 0$.
Thus, $\ker k\cap F(Z)=\mathbf{0}$.
Since $F(Z)$ has countable index in $F(X)$, $\ker k$ is also countable.
By lemma \ref{Le:1}, $\widehat{F(X)}/\overline{\hull{k}}$ is topologically isomorphic to $\widehat{\ker k}$; so it is metrizable.
Then $\widehat{F(X)}/\overline{N}$ is also metrizable as a quotient group, and so is its dense subgroup $K/N$, where $\overline{N}$ is the closure of $N$ in $\widehat{F(X)}$.\end{proof}
In summary, the groups $G$ and $H$ satisfy (i)--(iv).
\end{proof}

One may wonder what is the number of infinite metrizable quotient groups of $K$ in the above theorem.
The following result provides the answer: the number is $\mathfrak{c}$.
\begin{proposition}Every  pseudocompact abelian group $G$ with uncountable weight has exactly $w(G)^\omega$ many infinite metrizable quotient groups.\end{proposition}
\begin{proof}Let $\tau=w(G)$.
It suffices to prove that $G^*$ has exactly $\tau^\omega$ many closed countable subgroups.
Since $G^*$ has $\tau^\omega$ many countable subgroups, we shall complete the proof by showing every countable subgroup of $G^*$ is closed.
Let $H$ be a countable subgroup of $G^*$ and $g\in G^*$ be an element in the closure of $H$.
The the subgroup $H_1=H+\hull{g}$ is countable as well.
By Fact \label{Fac} (ii), the topology on $H_1$ inherited from $G^*$ is the Bohr topology.
Thus, $H$ is closed in $H_1$.
As also a dense subgroup of $H_1$, $H$ must coincides with $H_1$.
So $g\in H$.
In other words, $H$ is closed.
\end{proof}

\section{Comments and Problems}\label{Sec4}

Throughout this paper, we discuss the sets $QW(G)$ and $QD(G)$ for precompact abelian groups.
The result of Leiderman, Morris and Tkachenko shows that there exists a precompact abelian group $G$ such that $QD(G)\neq \varnothing$ and $\omega\notin QD(G)$ (so $\omega\notin QW(G))$.
Following this line, we find that under the assumption of $2^{<\mathfrak{c}}=\mathfrak{c}$, there exists a precompact abelian groups $G$ and $H$ such that $QD(G)\neq \varnothing\neq QW(H)$ and $$[\omega, \mathfrak{c})\cap QD(G)=\varnothing=[\omega, 2^{\mathfrak{c}})\cap QW(H).$$

However, the author does not know whether this remains valid in $\mathbf{ZFC}$ without any other set-theoretic assumption.

\begin{question} Do $\mathfrak{m}=\mathfrak{c}$ and $\mathfrak{n}=2^{\mathfrak{c}}$ hold in $\mathbf{ZFC}$?\end{question}

In Theorem \ref{Th:set1}, it is shown that the set $QW(G)$ can be rather arbitrary.
So it is natural to consider the same problem for $QD(G)$.
One may ask can we replace $QW(G)$ by $QD(G)$ in Theorem \ref{Th:set1}?
The question divides into two parts:
\begin{itemize}
\item [(i)] does every subset of $[\omega, \mathfrak{c}]$ equals $QD(G)$ for some precompact abelian group $G$?
\item [(ii)] if $QD(G)$ is a subset of $[\omega, \mathfrak{c}]$ and $\tau\in QW(G)$, can the cardinality of the set of quotient groups $K$ of $G$ with $d(K)=\tau$ be an arbitrary cardinal $\leq \mathfrak{c}$?
\end{itemize}

We shall see that the answers to both the two questions are ``no''.
To see this, let us first introduce an lemma which is ``folklore''.

\begin{lemma}\label{Le:Nov6}Let $G$ be a topological group and $N$ a closed subgroup.
If $X$ is a subset of $G$ such the image of $X$ in the coset space $G/N$ is dense and $Y$ is a dense subset of $N$, then $XY$ is a dense subset of $G$.
In particular, $d(G)\leq d(N)d(G/N)$.\end{lemma}
\begin{proof}Let $U$ be a non-empty open set in $G$.
Since the image of $X$ in $G/N$ is dense, $U\cap XN\neq \varnothing$.
So there exists $x\in X$ such that $U\cap xN$ is an non-empty open subset of $xN$.
Thus, $U\cap xN$ intersects the dense subset $xY$ of $xN$, i.e., $U\cap XY\neq \varnothing$.

To see the second assertion, one may assume that $Z$ is a dense subset of $G/N$ such that $d(G/N)=Z$ and let $X$ be a subset of $G$ such that $\pi(X)=Z$ and $\pi\res_X$ is injective, where $\pi:G\to G/N$ is the canonical projection.
Then the above argument ensures that $XY$ is dense in $G$ for each dense subset $Y$ of $N$.
Particularly, $Y$ can be chosen to have cardinality $d(N)$.
Thus $d(G)\leq |XY|\leq |X||Y|=d(G/N)d(N)$.\end{proof}

\begin{proposition} Let $G$ be a precompact abelian group with $d(G)=\tau>\omega$.
Then there exists $\tau$ many quotient groups $K$ of $G$ such that $d(K)=\tau$.\end{proposition}

\begin{proof}Note that since $d(G)=\tau$, every subset of $G$ with cardinality $<\tau$ generates a dense subgroup of a proper closed subgroup of $G$.
So, one may define by induction a subset $\{x_\alpha:\alpha<\tau\}$ of $G$ such that $x_\alpha$ is not in the closure $N_\alpha$ of $\hull{\{x_\beta:\beta<\alpha\}}$.
The inequality $d(N_\alpha)\leq \alpha+\omega<\tau$ holds for each $\alpha<\tau$ since $N_\alpha$ has a dense subgroup generated by $\alpha$ many elements.
Then by Lemma \ref{Le:Nov6}, each $G/N_\alpha$ has density $\tau$.\end{proof}

The proposition immediately implies the following facts.
\begin{corollary} If $E$ is a set of cardinalities such that $E=QD(G)$ for some precompact abelian group $G$, then $E$ has the largest possible element $d(G)$, so $\sup E\in E$.\end{corollary}

\begin{corollary} If $G$ is a precompact abelian group and $\omega\neq \tau\in QD(G)$, then $G$ has at least $\tau$ many quotient groups with density $\tau$.\end{corollary}

The following problem remains open:

\begin{problem}Characterize subsets $E$ of $[\omega, \mathfrak{c}]$ such that $E=QD(G)$ for some precompact abelian group $G$.\end{problem}

Let us come back to the separable quotient problem.
We say a topological group $G$ satisfying \emph{separable quotient theorem} (resp. \emph{metrizable quotient theorem}) if $G$ has an infinite quotient group which is separable (resp. metrizable).
In the case $G$ is precompact, metrizable quotient theorem implies separable quotient theorem.
We have known that pseudocompact groups satisfies metrizable quotient theorem, but in general, precompact abelian groups may fail to satisfy even separable quotient theorem.

Another important class of precompact abelian groups is the class of \emph{minimal} abelian groups \cite{DPS}. A topological group $G$ is called minimal if every continuous bijective homomorphism $\varphi: G\to H$ of topological groups is open, i.e., is a topological isomorphism.
Evidently, totally minimal groups are exactly topological groups with all quotient groups minimal.
It is known as the Prodanov-Stoyanov Theorem \cite{DPS} that every minimal abelian group is precompact.
We raise the separable (metrizable) quotient problem for minimal abelian groups.
\begin{question}\label{ques} Do minimal abelian groups satisfy separable (metrizable) quotient theorem?\end{question}
Note that we have provided a partial answer in the section of introduction that totally minimal abelian groups satisfies both the two theorems.

Now let us give a negative answer to the metrizable part of Question \ref{ques}, while, the rest part remains open.
\begin{lemma}\label{Le:mini}Let $G$ be a precompact abelian group, then
\begin{itemize}
  \item [(i)] $G$ is minimal if and only if $G^*$ contains no proper dense subgroup;
  \item [(ii)] $G$ is totally minimal if and only if every subgroup of $G^*$ is closed.
\end{itemize}
\end{lemma}
\begin{proof} (i) By the definition of minimal groups, a precompact abelian group $G$ is minimal if and only if it does not admit a strictly coarser group topology,
equivalently, if and only if $G^*$ does not admit a proper dense subgroup, by Proposition \ref{Prop:1.9}.

(ii) A precompact abelian group $G$ is totally minimal if and only if every quotient group of $G$ is minimal, which is equivalent to that no closed subgroup of $G^*$ contains a proper dense subgroup.
This is exactly the case that every subgroup of $G^*$ is closed.
\end{proof}

\begin{theorem} There exists a minimal abelian group $G$ of which every infinite metrizable quotient group has weight $\geq \mathfrak{c}$.\end{theorem}

\begin{proof} Let $\Z(4)$ be the cyclic group of order $4$ and $H$ the compact group $\Z(4)^\omega$.
Strengthen the topology of $H$ by letting all homomorphisms from $H$ to the cyclic group $\Z(2)$ of order $2$ continuous.
Let $N$ be the socle of $H$, i.e., $N=\{h\in H:2h=0\}$.
Since the kernel of every homomorphism of $H$ to $\Z(2)$ contains $N$, the new topology coincides with the original one on $N$.
In the following, $H$ is always assumed to have the new topology we just obtained.

Let $M$ be a proper subgroup of $H$.
Then $M$ is contained in a subgroup $M'$ of $H$ with index $2$ in $H$.
The choice of the topology on $H$ yields that $M'$ is closed, in particular, $M$ is not dense in $H$.
So $H$ does not admit a proper dense subgroup.
Let $K$ be an infinite closed subgroup of $H$.
Then $K\cap N$ is infinite.
Note that $N$ is compact, so $K\cap N$ is also compact and $|K\cap N|\geq \mathfrak{c}$.

Let $G=H^*$.
According to Lemma \ref{Le:mini}, $G$ is minimal.
Since every infinite closed subgroup of $H$ has cardinality $\geq \mathfrak{c}$, we know that every infinite quotient group of $G$ has weight $\geq \mathfrak{c}$.
\end{proof}
\appendix

\section{Separable Quotient Problem for Non-totally Disconnected Locally Compact Groups}

In this appendix, we give an answer to the following question raised by Leiderman, Morris and Tkachenko in \cite{LMT}, namely, the separable quotient problem for non-totally disconnected locally compact group.
\begin{question}\cite[Problem 1.29]{LMT}\label{lc} Does every non-totally disconnected locally compact group $G$ have a separable quotient group which is (i) non-trivial; (ii) infinite;
(iii) metrizable; or (iv) infinite metrizable?\end{question}

Recall that again every connected locally compact group admits a non-trivial separable metrizable quotient group.
However, not all locally compact groups satisfy separable (metrizable) quotient theorem.
Indeed, in \cite[Corollary 3.3]{Wil}, for any infinite cardinal $\tau$, Willis constructed a totally disconnected topologically simple group $G$ with an open compact subgroup $U$ such that $w(U)=\tau$.
Thus, if $\tau$ is large enough, i.e., $\tau>\mathfrak{c}$, then $U$ is not separable or metrizable.
Since $U$ is open in $G$, $G$ fails to be separable or metrizable as well.
Being also topologically simple, $G$ satisfies neither separable quotient theorem nor metrizable quotient theorem.

Since the connected component of the identity of a topological group is closed and normal, locally compact topologically simple groups are either connected or totally disconnected.
For the connected case, by the solution of Hilbert's fifth problem, such groups must be Lie; so they are metrizable and separable.
Therefore, a negative answer to Question \ref{lc} cannot be obtained by constructing a topologically simple group.
But still, we can use the construction of Willis.
Our following theorem answers Question \ref{lc} (i), so the whole question, in negative.
\begin{theorem}\label{thapp}For every infinite cardinal $\mu$, there exists a non-totally disconnected locally compact group $G$ of which all non-trivial quotient groups have local weight $\geq 2^\mu$ and density $\geq \mu$.\end{theorem}

Our main tool is the \emph{wreath product} of topological groups.
Let $G, H$ be groups, $X$ a set.
Let also $\varphi: G\times X\to X$ be a group action.
This action induces an action $\varphi'$ of $G$ on $H^X$ by automorphisms by letting
$g\cdot(h_x)=(h_{gx})$, i.e, the $x$-th coordinate of $g\cdot (h_x)$ is $h_{gx}$.
The semidirect product $H^X\rtimes_{\varphi'} G$ is denoted by $H\wr_\varphi G$ (or, briefly, $H\wr G$) and called the \emph{wreath product} of $H$ by $G$ (via $\varphi$).

We shall generalize this definition to topological groups, with the special case that $X$ is discrete.
When we say a topological space $X$ is a $G$-space for a topological group $G$, we mean there exists an action $\varphi: G\times X\to X$ by self-homeomorphisms which is a continuous with respect to the product topology.
\begin{theorem}\label{th1}Let $G, H$ be topological groups and $X$ a discrete $G$-space.
Then the wreath product $H\wr G=H^X\rtimes G$ with the natural product topology is a topological group.\end{theorem}
\begin{proof} In \cite[Proposition 7.2.1]{DPS}, a criterion for a semidirect product of topological groups $A$ by $B$ (with the product topology) is given:
it is precisely the case that the group action $B\times A\to A$ by conjugations is continuous.
Thus, we only need to prove that the action $\varphi':G\times H^X\to H^X$ is continuous, where $\varphi$ is the continuous action of $G$ on $X$ and $\varphi'$ is induced by $\varphi$ in the natural way.

To see that $\varphi'$ is continuous, we take a non-empty open subset $U$ of $H^X$ and $g\in G, (h_x)\in H^X$ such that $g\cdot (h_x)=(h_{gx})\in U$.
Note that $H^X$ has the product topology.
So we may assume without loss of generality that $U=V^y:=\{(h_x)\in H^X, h_y\in V\}$, where $V$ is open in $H$ and $y\in X$, since the family $\{V^x: V\subseteq H~\mbox{open and~} x\in X\}$ forms a subbase of $H^X$.

Let $y'=g^{-1}y$ and $U'=V^{y'}$.
Since $X$ is discrete, there exists an open neighborhood $O$ of $g$ such that $g'y'=y$ for all $g'\in O$.
Thus, $$g'\cdot U'=g'\cdot V^{y'}=V^{g'y'}=V^y=U.$$
In other words, $\varphi'(O\times U')\subseteq U$.
So, $\varphi'$ is continuous.
\end{proof}

Recall that a connected, compact, and topologically simple group must be simple as an abstract group \cite[Theorem 9.9.6]{HM0}.
So we prefer to call such groups simple for convenience and there will be no confusion.
In the following, we shall see that closed normal subgroups of powers of connected, compact, simple groups are quite ``regular''.
For the sake of simplicity, for a product group $G=\prod_{x\in X}N_x$ and a subset $Y$ of $X$, we shall consider $\prod_{y\in Y}N_y$, in particular, each $N_x$, as a subgroup of $G$ in a natural way.

\begin{lemma}\label{leapp}Let $N$ be a non-trivial connected, compact, simple group and $X$ be any set.
Let also $G=\prod_{x\in X}N_x$, where each $N_x$ is a copy of $N$.
Then for every non-trivial closed normal subgroup $H$ of $G$, there exists a subset $Y$ such that $H=\prod_{y\in Y}N_y$.\end{lemma}

\begin{proof}Let $p_x: G\to N_x$ be the canonical projection for each $x\in X$.
Since $p_x(H)$ is a normal subgroup of the simple group $N_x$, $H_x:=p_x(H)$ equals either $N_x$ or $\mathbf{1}$.
Let us denote by $Y$ the subset of $X$ consisting all the elements $x\in X$ with $H_x=N_x$.
Then we have
\begin{equation}\label{e1} H\subseteq \prod_{y\in Y}N_y.\end{equation}
Clearly, $Y$ is non-empty since $H$ is non-trivial.

For each $y\in Y$, we claim that $N_y\subset H$, which is equivalent to that $N_y\cap H\neq \mathbf{1}$.
Suppose for a contradiction that there exists $y\in Y$ such that $N_y$ trivially meets $H$.
Note that non-trivial connected compact simple groups are non-abelian.
So there exists $a, b\in N_y$ such that $aba^{-1}b^{-1}\neq 1$.
Take $h\in H$ with $p_y(h)=b$.
Since both $N_y$ and $H$ are normal in $G$, one has that
$$aha^{-1}h^{-1}\in N_y\cap H=\mathbf{1}.$$
Thus, we have that $1=p_y(aha^{-1}h^{-1})=aba^{-1}b^{-1}$, which induces a contradiction.
Now we proved the claim that each $N_y$ is contained in $H$.
Since $\{N_y:y\in Y\}$ generates a dense subgroup of $\prod_{y\in Y}N_y$ and $H$ is compact, the inclusion
\begin{equation}\label{e2} \prod_{y\in Y}N_y\subseteq H\end{equation}
holds.
By (\ref{e1}) and (\ref{e2}), we complete the proof.
\end{proof}
\begin{theorem}\label{th2}Let $N$ be a non-trivial connected, compact, simple group, $G$ a topologically simple group with a proper open subgroup $H$ and $X=G/H$ the (discrete) left coset space.
Then the wreath product $K=N\wr G=N^X\rtimes G$ via the left translation of $G$ on $X$ is a topological group.
Moreover $\mathbf{1}$, $N^X$ and $K$ are the only closed normal subgroups of $K$.\end{theorem}
\begin{proof}
The assertion that $K$ is a topological group follows from Theorem \ref{th1} immediately since the left translation $G\times X\to X$ is continuous.

Let $M$ be a closed normal subgroup of $K$.
By Lemma \ref{leapp}, there exists a subset $Y$ of $X$ such that $M\cap N^X=N^Y$ ($N^\varnothing$ is considered as $\mathbf{1}$).
We claim that $Y$ is either $\varnothing$ or the whole $X$.
Suppose for the contrary that there are $x\in X\setminus Y$ and $y\in Y$.
Take $g\in G$ with $gx=y$.
Let $\alpha=(1, g)\in N^X\rtimes G=K$.
Since $N^Y$ is the intersection of normal subgroups of $K$, it is normal in $K$ as well.
Thus $$N^Y=\alpha N^Y \alpha^{-1}=g\cdot N^Y=N^{gY}.$$
So $Y=gY$ and therefore $x=g^{-1}y\in g^{-1}Y=Y$.
This contradicts to the choice of $x$ and $y$.

Now the problem divides into the two cases:

\noindent{\bf Case 1.} $Y=X$.

When in this case, $M$ contains $N^X$.
Since $N^X$ is compact, $M/(N^X)$ is a closed normal subgroup of the topologically simple group $K/(N^X)\cong G$;
thus either $M=N^X$ or $M=G$.

\noindent{\bf Case 2.} $Y=\varnothing$, i.e., $M\cap N^X=\mathbf{1}$.

Let $\pi: K\to K/(N^X)\cong G$ be the canonical quotient mapping.
Then $\pi$ is closed since $N^X$ is compact and therefore $\pi(M)$ is a closed normal subgroup of $G$.
Note that the restriction of $\pi$ to $M$ is injective since $M\cap N^X$ is trivial.
So if $\pi(M)$ is trivial, then also is $M$.
If $\pi(M)$ is non-trivial, it must equals $G$.
Then the subgroup $(N^X)M$ of $K$ equals $K$.
Note that $N^X$ and $M$ are normal subgroups of $K$, $K$ is the direct product of $N^X$ and $M$.
This implies that $N^Z$ is normal in $K$ for any subset $Z$ of $X$.
Since $|X|\geq 2$, $Z$ can be chosen to be distinct from either $\varnothing$ or $X$.
While, this is a contradiction since our analysis yields that closed normal subgroups of $K$ contained in $N^X$ equals either $N^X$ or $\mathbf{1}$.
Thus, $\pi(M)=G$ is impossible in this case.

In summary, we obtain that $M=N^X$, $G$ or $\mathbf{1}$.
\end{proof}

\begin{proof}[\bf Proof of Theorem \ref{thapp}]
For a topological group $X$, we denote by $\chi(X)$ the local weight of $X$ (at the identity).
By the result of Willis, there exists a totally disconnected, locally compact, topologically simple group $G$ with an open compact subgroup $H$ such that $w(H)=\tau$ with $\ln\tau\geq \mu$.
Let $X=G/H$ and $N$ a non-trivial connected, compact, simple group.
The non-totally disconnected locally compact group $K=N\wr G$ defined as in Theorem \ref{th2} therefore has only three closed normal subgroups: $K$, $N^X$ and $\mathbf{1}$.
So $K$ admits only two non-trivial quotient groups: $K$ and $K/M^X\cong G$.
Since $G$ is a quotient group of $K$, one has $\chi(K)\geq \chi(G)$ as well as $d(K)\geq d(G)$.
Moreover, since $H$ is open in $G$, one has $\chi(G)=\chi(H)=w(H)$ and $d(G)=d(H)$.
It remains to check that $\chi(H)=\tau\geq 2^\mu$ and $d(H)\geq \mu$.
The first inequality comes from $\ln \tau\geq \mu$.
To see the second, one should be noted that for every dense subset $X$ of $G$, $X\cap H$ is dense in $H$ since $H$ is open.
Hence $d(G)\geq d(H)$.
The the assertion follows from immediately the fact that every compact group $H$ satisfies $d(H)=\ln w(H)$.
\end{proof}

\section*{Acknowledgements}
The author appreciates very much the professional remarks from the referee which improve the paper a lot.
In a very early version of this paper, the bad using of notions brings confusions. The author is grateful to Professor Arkady Leiderman for pointing out that.
The author would like also to thank Professor Wei He for a number of helpful suggestions, and to Zhouxiang Huang, Haonan Qian, V\'{\i}ctor Hugo Ya\~{n}ez and Gao Zhang for listening to the report of the main results of this paper in a seminar.

\end{document}